\documentclass[a4paper,12pt,twoside]{article}

\usepackage[body={170mm,240mm}]{geometry}
\usepackage{a1}
\usepackage[english]{babel}
\usepackage[latin1]{inputenc} 
\usepackage{amsfonts,amssymb}
\usepackage{amsmath}
\usepackage{graphicx}	

\usepackage{makeidx}
\usepackage{xcolor,colortbl}
\usepackage{array}
\makeindex

\definecolor{Gray}{gray}{0.85}
\newcolumntype{d}{>{\columncolor{Gray}}c}
\newcolumntype{b}{>{\columncolor{white}}c}
\newcolumntype{a}{>{\columncolor{lightgray}}c}
\newcolumntype{?}{!{\vrule width 2pt}}

\def\mapright#1#2#3{\smash{\mathop{\hbox to
#3{\rightarrowfill}}\limits^{#1}_{#2}}}

\def\mapleft#1#2#3{\smash{\mathop{\hbox to
#3{\leftarrowfill}}\limits^{#1}_{#2}}}

\def\mapright#1#2{\smash{\mathop{\hbox to 0.90cm{\rightarrowfill}}\limits^{#1}_{#2}}}
\def\mapleft#1#2{\smash{\mathop{\hbox to 0.90cm{\leftarrowfill}}\limits^{#1}_{#2}}}

\def\mapleftright#1#2{\smash{\mathop{\hbox to 0.80cm{\leftarrowfill \rightarrowfill}}\limits^{#1}_{#2}}}

\title{An $O(|E|)$-linear Model for the MaxCut Problem
}
\author{S\'ostenes L. Lins and Diogo B. Henriques \\ CIn/UFPE/Brazil}
\date{\today}

\date{\today}

\begin{document}

\maketitle

\begin{abstract}
A polytope $P$ is a {\em model} for a combinatorial problem on finite graphs
$G$ whose variables are indexed by the edge set $E$ of $G$ if the points of $P$ with
(0,1)-coordinates are precisely the characteristic vectors of the subset of edges
inducing the feasible configurations for the problem. In the case of the (simple) MaxCut 
Problem, which is the one that concern us here,
the feasible subsets of edges are the ones inducing
the bipartite subgraphs of $G$.
In this paper we introduce a new polytope $\mathbb{P}_{12} \subset \mathbb{R}^{|E|}$ given 
by at most $11|E|$ inequalities, which is a model for the 
MaxCut Problem on $G$. Moreover, the left side 
of each inequality is the sum of at most 4 edge variables with coefficients $\pm1$ 
and right side 0,1, or 2. We restrict our analysis
to the case of $G=K_{z}$, the complete graph in $z$ vertices, where $z$ is
an even positive integer $z\ge 4$. This case is sufficient to study 
because the simple MaxCut problem for general graphs $G$ can be reduced 
to the complete graph $K_z$ by considering the obective function of the associated 
integer programming
as the characteristic vector of the edges in $G \subseteq K_z$. 
This is a polynomial algorithmic transformation.
\end{abstract}

\section{Notation and Preliminaries} 
The {\em MaxCut Problem} \cite{garey1979computers} is one of the first NP-complete problems.
This problem can be stated as follows. Given a graph $G$ does it has a bipartite subgraph
with $n$ edges? It is a very special problem which has been acting as a 
paradigm for great theoretical
developments. See, for instance \cite{goemans1995iaa}, where an algorithm with a 
rather peculiar worse case
performance (greater than 87\%) can be established as a 
fraction of type (solution found/optimum solution). This result constitutes a
landmark in the theory of approximation algorithms. 

Our approach is a theoretical investigation
on polytopes associated to complete graphs. The main result is that there is a set of 
at most $11|E|$ short inequalities (each involving no more than 4 edge variables 
with coefficients $\pm1$) so that
the polytope in $\mathbb{R}^{|E|}$ formed by these inequalities 
has its all integer coordinate points in 1-1 correpondence
with the characteristic vectors of the complete bipartite subgraphs of $K_z, z$ even. 
 
\vspace{3mm}
\noindent
{\bf Thick graphs into closed surfaces.} A surface is {\em closed} 
if it is compact and has no boundary. 
A closed surface is characterized by its Euler characterisitic and the 
information whether or not is orientable.
We use the following 
combinatorial 
counterpart for a graph $G$ cellularly embedded into a closed surface $S$, 
here called a {\em map}. 
{\em Cellularly embedded} means that $S\backslash G$ is a finite set of open disks
each one named a {\em face} of the embedding, whence a {\em surface dual graph}
is well defined. Each edge is replaced in 
the surface by an $\epsilon$-thick version of it, named $\epsilon$-rectangle.
Each vertex $v$ is replaced by a $\delta$-disk, where $\delta$ is the radius of the 
disk whose center is $v$. The $\epsilon$-rectangles and 
the $\delta$-disks form the {\em thick graph of $G$}, denoted by
$T(G)$. By choosing an adequate pair $(\epsilon < \delta)$, the boundary of $T(G)$ 
is a {\em cubic graph} (i.e., regular graph of degree 3), denoted by $C(G)$.
The edges of $C(G)$ can be properly colored with 3 colors: we have {\em short}, 
{\em long}, and {\em angular} colored edges so that at each 
vertex of $C(G)$ the three colors appear. The long (resp. short)
colored edges are the edges
which induced by the long (resp. short) sides of 
the $\epsilon$-rectangles. The angular edges are the other edges.

\vspace{3mm}
\noindent
{\bf Gems or hollow thick graphs.} A cubic 3-edge colored graph $H$ in 
colors $(0,1,2)$ is called a {\em gem} (for {\em g}raph-{\em e}ncoded 
{\em m}ap) if the 
connected components induced by edges of colors
$0$ and $2$ are {\em polygons} with 4 edges. A {\em polygon} in a graph is 
a non-empty subgraph which is connected and has each vertex of degree 2.
A {\em bigon} in $H$ is a connected component of the subgraph induced by all
the edges of any two chosen chosen among the three colors. An $ij$-gon is a
bigon in colors $i$ and $j$.
From $H$ we can easily produce the surface $S$ and 
$G \hookrightarrow S$: attach 
disks to the bigons of $H$ thus obtaining $T(G) \hookrightarrow S$ 
up to isotopy. To get $G$ embedded into $S$ just contract the 
$\delta$-disks to points. Each rectangle becomes a digon and contracting
these digons to their medial lines we get $G \hookrightarrow S$. 
The Euler characteristic of $S$ is $v(H)+f(H)-r(H)$, where
$v(H)$ is the number of 01-gons of 
$G$ (or the number of vertices of $G$), $f(H)$ is the 
number of 12-gons of $H$
(or the number of faces of $G \hookrightarrow S$)
and $r(H)$ is the number of rectangles of $H$ (or the number of 
edges  of $H$). Moreover, $S$ is 
an orientable surface iff and only $H$ is a bipartite graph, see \cite{Lins1980}.
Note that in each gem any edge appear exactly in two bigons: indeed, if the 
edge is of color $i$ it will appear once in a $ij$-gon and once in a $ik$-gon,
where $\{i,j,k\}=\{0,1,2\}$. The surface of a map 
is obtainable 
from the gem by attaching disks to the bigons and identifying the boundaries
along the two occurrences of each edges.

\begin{figure}
\begin{center}
\label{fig:thickening}
\includegraphics[width=12cm]{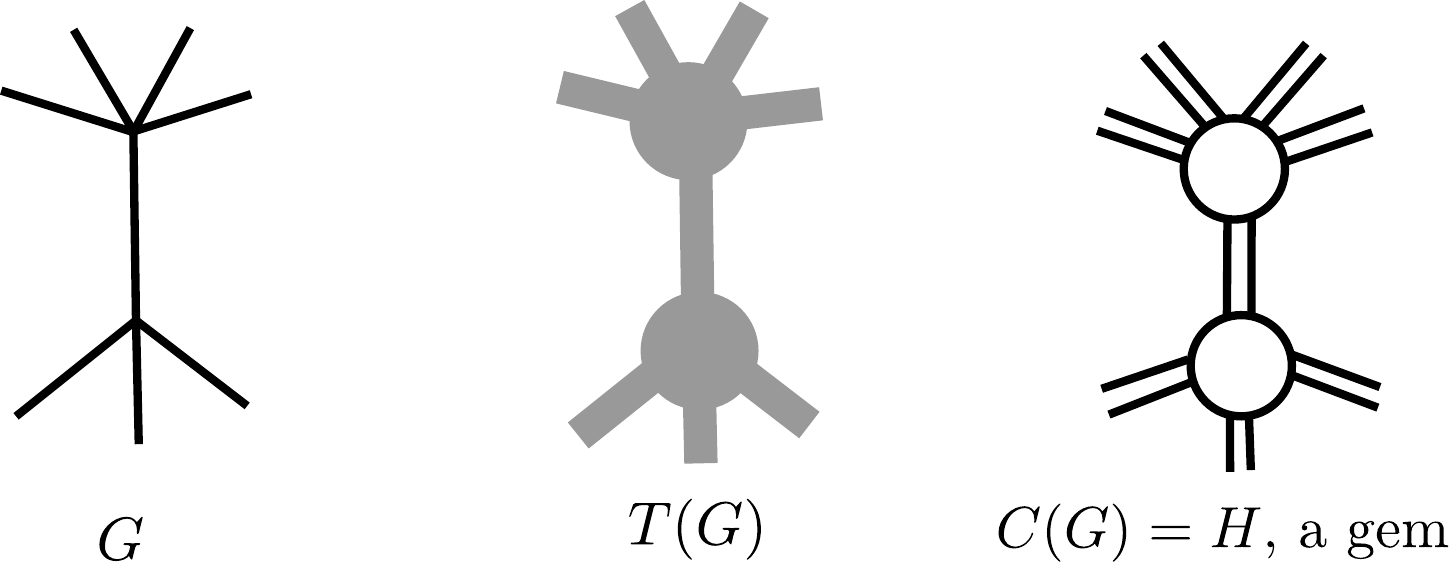} \\
\caption{\sf Neighborhood of an edge in $G \hookrightarrow S$ its thickened version 
and a hollow counterpart: the gem $H$. To get a $Q$-graph from $H$, let $\mu_0$
be the short edges, $\mu_1$ be the angular edges 
(they correspond to angles in $G$), let $\mu_2$ be the long edges, and finally
add the crossing edges $\mu_3$ as the diagonals of the $02$-rectangles of $H$.
Note that the colors of the edges of $Q$ are implicit $0$ is the color of the 
short edges of the 02-rectangles, $2$ of their long edges, $1$ is the color of the 
edges of $H$ not in the rectangles and $3$ is the color of the crossing edges.
} 
\label{fig:thickening}
\end{center}
\end{figure}

\vspace{3mm}
\noindent
{\bf $Q$-graphs and their dualities.}
A {\em perfect matching} in a graph with an even number, $v$, of vertices
is a set of $v/2$ pairwise disjoint edges. A {\em $Q$-graph $Q(\mu_0,\mu_1,
\mu_2,\mu_3)$} is the disjoint union of
4 ordered of its perfect perfect matchings 
$\mu_i, \ i=0,1,2,3$, so that each component of
$\mu_0 \cup \mu_2 \cup \mu_3$ is a complete graph $K_4$. Each such $K_4$ is
called a {\em hyperedge} of the $Q$-graph. The edges in $\mu_1$
are called {\em angular edges} of the $q$-graph. 
The edges in $\mu_0$ are called {\em short edges},
the ones in $\mu_2$, {\em long edges}, the ones in $\mu_3$ are called the
{\em crossing edges}. The graphs $Q(\mu_0,\mu_1,\mu_2,\mu_3)$ and 
$Q(\mu_2,\mu_1,\mu_0,\mu_3)$ are {\em dual $Q$-graphs}.  
The graphs $Q(\mu_0,\mu_1,\mu_2,\mu_3)$ and 
$Q(\mu_3,\mu_1,\mu_2,\mu_0)$ are {\em phial $Q$-graphs}.
The graphs $Q(\mu_0,\mu_1,\mu_2,\mu_3)$ and 
$Q(\mu_0,\mu_1,\mu_3,\mu_2)$ are {\em skew $Q$-graphs}.
To obtain a gem $H$, whence $G$ from a $Q$-graph, 
just remove its last perfect matching. Note that dual $Q$-graphs induce the same
surface $S$ and the same zigzag paths while interchanging 
boundary of faces and coboundaries of vertices. Skew $Q$-graphs induce 
the same graph $G$ and interchange coboundary of faces and zigzag paths.
Phial $Q$-graphs interchange coboundary of vertices and 
zigzag paths while maintaining the 
boundaries of the faces (as cyclic set of edges) in the respective surfaces, see
Fig. \ref{fig:Qdualities}.
Note that the embedding $G \hookrightarrow S$ defines the $Q$-graph.
This enable us to identify 

\begin{center}
\begin{tabular}{c}
$Q_1 = Q(\mu_0,\mu_1,\mu_2,\mu_3) \equiv G_1 \hookrightarrow S^{12} \equiv G_1,$\\
$Q_2 = Q(\mu_2,\mu_1,\mu_0,\mu_3) \equiv G_2 \hookrightarrow S^{12} \equiv G_2,$\\
$Q_3 = Q(\mu_3,\mu_1,\mu_0,\mu_2) \equiv G_2^\sim \hookrightarrow S^{23} \equiv G^\sim_2$,\\
$Q_4 = Q(\mu_3,\mu_1,\mu_2,\mu_0) \equiv G_3 \hookrightarrow S^{23} \equiv G_3$,\\
$Q_5 = Q(\mu_2,\mu_1,\mu_3,\mu_0) \equiv G_3^\sim \hookrightarrow S^{31} \equiv G_3^\sim$, \\
$Q_6 = Q(\mu_0,\mu_1,\mu_3,\mu_2) \equiv G_1^\sim \hookrightarrow S^{31} \equiv G_1^\sim$.
\end{tabular} 
\end{center}

$G_1$ the graph of the dual map, $G_2$ and graph of the phial map $G_3$. 
To get the {\em phial} of a map, we interchange the short edges 
of the rectangles by their diagonals.
There are also the twisted maps $G_1^\sim$, $G_2^\sim$ and $G_3^\sim$. 
There are three closed surfaces
$S^{12}$ where $G_1$ and $G_2$ embed as duals, $S^{23}$ where $G_2^\sim$ and $G_3$
embed as duals and $S^{31}$ where $G_3^\sim$ and $G_1^\sim$ embed as duals.
For the case that concerns us, $G_3$ is $K_z$ with line embedding in $S^{12}$,
$G_1$ is $Pog_h$ and $G_2$ is the $\mathbb{RP}^2$-dual of $Pog_h$, since 
$S^{12}$ is $\mathbb{RP}^2$.  
These dualities were introduced first in \cite{Lins1980} and then in \cite{Lins1982}.

\begin{figure}
\begin{center}
\includegraphics[width=17cm]{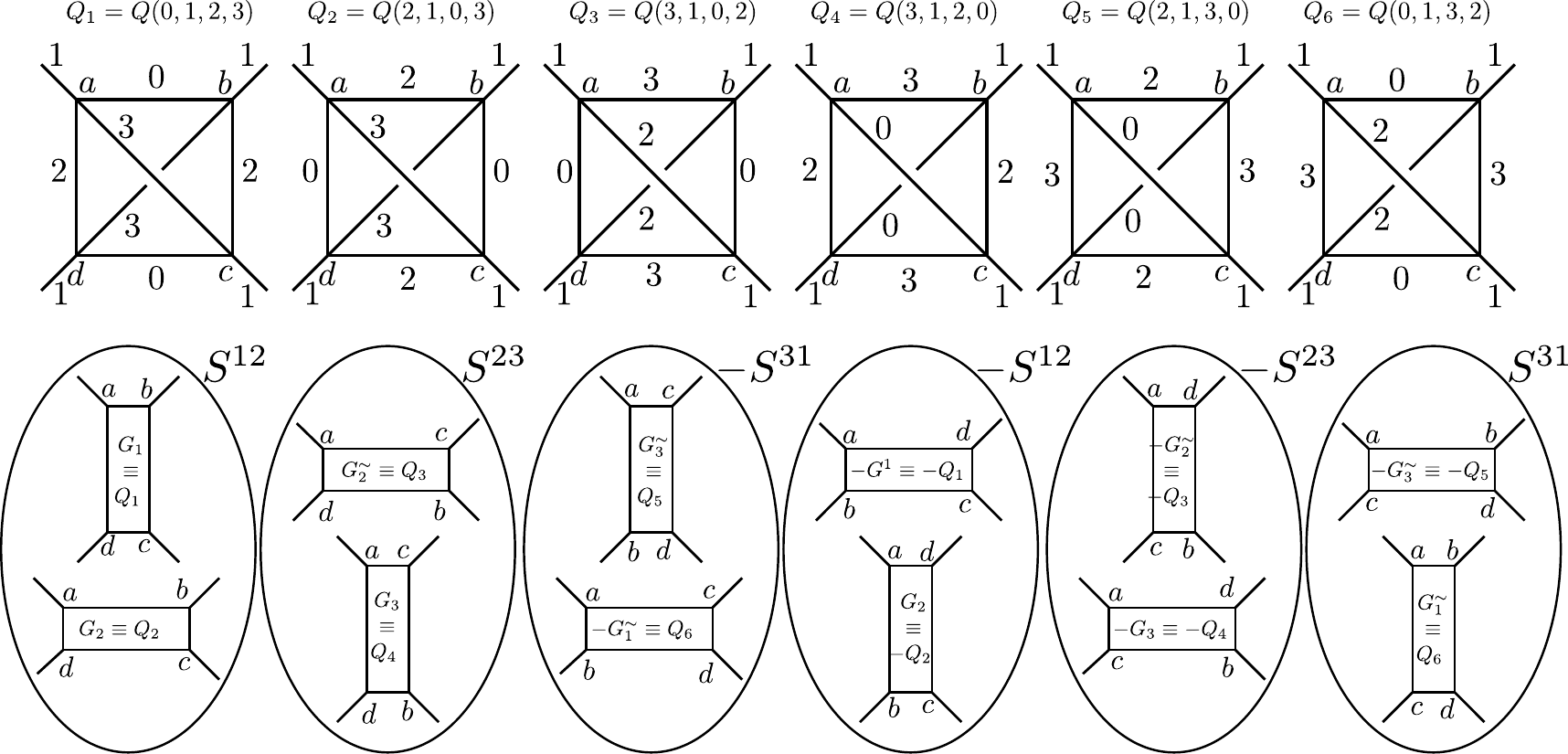}
\caption{\sf $Q$-graph $Q(h,i,j,k)$ is a short form of $Q(\mu_h,\mu_i,\mu_j,\mu_k)$.
We depict the $Q$-dualities of a $Q$-graph (usual surface duality, 
skew duality and phial duality)
which induce 3 graphs $G_1, G_2, G_3$ 
and 3 surfaces: $S^{12}, S^{23},S^{31}$. The minus signs mean a 
local reversal of orientation given by the cyclic order of the rectangle corners
$(a,b,c,d)$. Graphs $G_i$ and $G_i^\sim$ $(i=1,2,3)$ 
are the same: they are just embedded into distinct surfaces in such a way that
the faces of one are the zigzags paths (left-right paths) of the other. The zigzags paths are 
closed and well defined --- they correspond to the 13-gons of $Q$ --- even if the surface is
non-orientable, where is impossible to define left or right globally. 
Taking the dual $(\mathcal{DU})$  
corresponds in the gem to switch the vertical rectangles to horizontal ones (and vice-versa) 
while maintaining the cyclic order 
of the corners of the rectangle (so the surface does not change). 
Taking the skew  $(\mathcal{SK})$  corrresponds to exchange corners linked by one of 
the  two short sides of the rectangles. Starting with $Q(0,1,2,3)$ and applying 
iteratively the composition $\mathcal{SK} \circ \mathcal{DU}$ we get the six
$Q$-graphs which appear in the top of each one of the six surfaces. Taking the 
phial  $(\mathcal{PH})$ is defined as  
$\mathcal{PH}=\mathcal{DU} \circ \mathcal{SK} \circ \mathcal{DU}$, or 
directly by exchanging a pair of corners linked by one of the 2 long 
sides of the rectangle.
If we care for orientation
and all the 3 surfaces are orientable, then there are in fact 12 $Q$-graphs and 6 oriented surfaces. 
Note that $\mathcal{PH}=\mathcal{SK} \circ \mathcal{DU} \circ \mathcal{SK} \circ \mathcal{DU} 
\circ \mathcal{SK} \circ \mathcal{DU} \circ \mathcal{SK} \circ  \mathcal{DU} \circ \mathcal{SK}$.
But orientation does not concern us here. Therefore there are only 3 pairs of skew maps, each pair
inducing the same graph 
$\{G_1: \hookrightarrow S^{12},G_1^\sim: \hookrightarrow S^{31}\}$,
$\{G_2: \hookrightarrow S^{12},G_2^\sim: \hookrightarrow S^{23}\}$,
$\{G_3: \hookrightarrow S^{23},G_3^\sim: \hookrightarrow S^{31}\}$,
and 3 surfaces $S^{12}, S^{23}, S^{31}$.
}  \label{fig:Qdualities}
\end{center}
\end{figure}

\section{Reformulation of the MaxCut Problem}
\vspace{3mm}
\noindent

Let $G_1$ be an arbitrary map of a graph into a surface, orientable or not, and
$G_2, G_3$, denote respectively the dual and phial of $G_1$.
Let $E$ denote the common set of edges for graphs $G_1, G_2, G_3$: they are identified
via the hyperedges of the associated $Q$-graph.

\vspace{3mm}
\noindent
{\bf Vector spaces from graphs.} For subset of edges $A$ 
and $B$ let $A + B$ denote their symmetric difference. 
This is closely related with the sum in $GF(2)$ via the characterisitc vectors.
Thus an element is in $$A_1+ A_2 + \ldots + A_p$$ if it belongs
to an odd number of $A_i$'s. This sum on subsets of edges
becomes an associative binary operation and 
$2^E$, the set of all subsets of $E$, 
becomes a vector space via $+$ on subsets, or, what amounts to be the same,
the mod 2 sum of characteristic vectors of the subsets of edges.
There is a distinguished basis given by the 
characteristic vectors of the singletons. We say that subset of edges
$A$ is orthogonal to subset of edges $B$ if $|A \cap B|$ is even. 
If $W \subseteq 2^E$ is a subspace, then 
$W^\perp = \{u \in 2^E : u \perp w, \forall w\in W\}$ is also a subspace and 
$dim W + dim\ W^\perp = |E|$. Let $\mathcal{V}_i$ $(i=1,2,3)$ 
be the subspace of $2^E$ generated
by the coboundary of the vertices of $G_i$, or {\em coboundary space} of $G_i$.
The {\em cycle space} of $G_i$ is  $\mathcal{V}_i^\perp$. The {\em face space}
of $G_i$, denoted by $\mathcal{F}_i$, 
is the subspace of $\mathcal{V}_i^\perp$ generated by the face boundaries
of $G_i$.  The {\em zigzag space}
of $G_i$, denoted by $\mathcal{Z}_i$, 
is the subspace of $\mathcal{V}_i^\perp$ generated by the zigzag paths
of $G_i$. Note that $G_i$ is rich iff
$\mathcal{V}_i^\perp$=$\mathcal{F}_i$+$\mathcal{Z}_i$. In particular,
$\mathcal{F}_1 = \mathcal{V}_2$ and $\mathcal{Z}_1=\mathcal{V}_3$.

\begin{theorem}[Absortion property] 
Let $(i,j,k)$ denote a permutation of $\{1,2,3\}$.
Then  $\mathcal{V}_i \cap \mathcal{V}_j \subseteq \mathcal{V}_k$.
\end{theorem}
\begin{proof}
For a proof we refer to Theorem 2.5 of \cite{Lins1980}. The proof is long and 
we do not know a short one. This is a basic property which opens the way for a
perfect abstract symmetry among vertices, faces and zigzags.
A useful consequence of
this property is that $
\mathcal{V}_1 \cap \mathcal{V}_2 = 
\mathcal{V}_1 \cap \mathcal{V}_3 = 
\mathcal{V}_2 \cup \mathcal{V}_3 = 
\mathcal{V}_1 \cap \mathcal{V}_2 \cup \mathcal{V}_3$.
\end{proof}

The {\em cycle deficiency} of 
$G_i$ is $cdef(G_i)=dim((\mathcal{V}_i^\perp) / (\mathcal{V}_j+\mathcal{V}_k))$.
Map $G_i$ is {\em rich} if its cycle deficiency is 0, implying, in particular,
$\mathcal{V}_i=\mathcal{V}_j^\perp \cap \mathcal{V}_k^\perp$, for all permutations
$(i,j,k)$ of $(1,2,3)$.

\begin{corollary}
\label{cor:deficity}
Maps $G_1, G_2, G_3$ have the same cycle deficiency.
\end{corollary}
\begin{proof}
Assume $G_1$ has $e$ edges, $v$ vertices, $f$ faces and $z$ zigzags.
Then  $$cdef(G_1)=(e-v+1)-((f-1)+(g-1)-\gamma) = e-(v+f+g)+(3+\gamma),$$ 
where $\gamma=dim(\mathcal{V}_1 \cap \mathcal{V}_2 \cup \mathcal{V}_3).$ 
The Corollary follows because $v+f+z$ is invariant under permutations of $(v,f,z)$.
\end{proof}

Thus richness is a symmetric property on the maps $G_1$, $G_2$, $G_3$: we have $G_1$ 
is rich $\Leftrightarrow G_2$ is rich $\Leftrightarrow G_3$ is rich.
A subgraph is {\em even} if each of its vertices has even degree. 

\begin{corollary}
\label{cor:even}
If $F \subseteq E$ induces an even subgraph of $G_i$, then $F \in \mathcal{V}_i^\perp$.
\begin{proof}
Any polygon of $G_i$ is in $\mathcal{V}_i^\perp$. Note that 
$F$ is a sum of polygons and so, $F \in \mathcal{V}_i^\perp$.
\end{proof}
\end{corollary}

A subset $F\subseteq E$ is a {\em strong $O$-join in $G_1$} if it induces a subgraph 
so that at each vertex $v$ and each face $f$ the parity of the 
number of $F$-edges in the coboundary of $v$ and in the boundary of $f$
coincides with the parity of the degrees of $v$ and $f$, respectively. 
Note that $F$ is a strong $O$-join iff 
$\overline{F}=E\backslash F \in \mathcal{V}_1^\perp \cap \mathcal{V}_2^\perp$.
See Fig. \ref{fig:rozig4}, where we depict a strong $O$-join $T$ given by the thick edges
in $G_1=Pog_4$. In the case of a rich $G_3$, $F$ is a 
strong $O$-join of $G_1$ iff $\overline{F} \in \mathcal{V}_3$.

The {\em coboundary of a set of vertices $W$} is the set of edges which has
one end $W$ and the other in $V\backslash W$.
A subset of edges is a {\em coboundary in a graph} iff it induces a bipartite
subgraph: the edges of this bipartite graph 
constitutes the coboundary of the set of vertices in the same class of the bipartition.
A {\em cut} in combinatorics is frequently defined as a minimal coboundary. Thus it is 
preferable to talk about maximum coboundary instead of talking about maximum cut to avoid
misunderstading.

\begin{theorem}[Reformulation of MaxCut problem]
 Let $G_3$ be a rich map. The maximum cardinality of a coboundary in $G_3$
 has cardinality equal to $|E|$ minus the minimum cardinality 
 of a strong $O$-join in $G_1$.
 \begin{proof}
 The result follows because the complement of a strong $O$-join $F$ 
 is an even subgraph 
 in graphs $G_1$ and $G_2$. Thus, $\overline{F} \in 
 \mathcal{V}_1^\perp \cap \mathcal{V}_2^\perp=\mathcal{V}_3$. 
  The last equality follows because $G_3$ is rich. Note that the elements 
  of $\mathcal{V}_3$ are 
  precisely the coboundaries of $G_3$.
 \end{proof}
\end{theorem}


\section{Projective Orbital Graphs}

{\bf Motivation to restrict to $G_3=K_z, z$ even}.
In order to use the $Q$-dualities and rich maps 
we must start with a rich map $G_3$. Our universal choice for
$G_3$ is the complete graph $K_z$ with $z$ even. There are various reasons for this choice.
(a) Every graph is a subgraph of some $K_z$. (b) It is very easy to embed $G_3=K_z$ in some surface 
so that its phial $G_1$ and dual of the phial $G_2$ are embedded into the real projective plane, 
$\mathbb{PR}^2$: the simplest
closed surface after the sphere. (c) There is a combinatorial
well structured 
generator subset of the cycle space of $K_z$,
$\mathcal{V}_3$, given by all but one coboundaries of the vertices of $G_1$ and the all but one 
coboundaries of the vertices of $G_2$ (faces of $G_1$). 
Moreover each one of these generators correspond
to a polygon in $K_z$ having either 3 or 4 edges. Finally, (d) the maximum cardinality of
a bipartite subgraph of an arbitrary graph $G$ with $z$ vertices can be obtained by solving the 
integer 0-1 programming problem 
using the characteristic vector of the edge set of $G$ relative to the complete graph $G_3=K_z$ as
objective function. If $z$ is odd attach a pendant edge $e$ to $G$, and solve the problem for 
$G+e \subset K_{z+1}$.
All these properties justify the restriction to complete graphs with an even number of vertices.

Our model will use inequalities induced
signed forms of these generating polygons 
in all possible ways. So it is paramount to have short
polygons as generators, otherwise an exponential number of 
inequalities arises from the beginning. Our approach
starts by constructing graph $G_1=Pog_h$ embedded into $\mathbb{RP}_2$, 
and its description follows.

\vspace{3mm}
\noindent
{\bf The projective orbital graphs}. 
Let $h \in \{\ 1,\frac{3}{2},2,\frac{5}{2},3,\frac{7}{2},4,\frac{9}{2},\ldots\}$.
The {\em Projective orbital graph or $Pog(h)$} is defined as follows. 

\noindent
{\bf Case $h$ integer.}
If $h$ is an integer, then $Pog(h)$ consists
of $h$ concentric circles (orbits) having each $z=4h$ vertices equally spaced. In the 
complex plane the $hz$ vertices of $Pog(h)$ are 
$\{k \exp(2\pi i j/z): k=1,2,\ldots, h, \hspace{2mm} j=0,\ldots, z-1\}.$
Each one of the $h$ orbits of $Pog(h)$
induces $z$ edges as closed line segments in the complex plane: 
$$\{{[k \exp(2\pi i j/z,k \exp(2\pi i (j+1)/z})]): j=1,\dots,z \}.$$ 
These edges are called {\em orbital edges}.
There are also $zh$ {\em radial segments} being $z(h-1)$ {\em radial edges} and $z$ {\em pre-edges}:
$\{{[k \exp(2\pi i j/z),(k+1)\exp(2\pi i j/z}]): k=1,\dots,h,\hspace{2mm} j=1,\ldots,z \}.$
Note that the $z$ points $\{{[(h+1)\exp(2\pi i j/z}]): j=1,\ldots,z \}$ are not vertices of $Pog(h)$
and are called {\em auxiliary points}. Each one of the radial segments
incident to an auxiliary point is a {\em pre-edge}. The graph whose vertices are the vertices of $Pog(h)$
plus the auxiliary points and whose edges are the edges plus pre-edges of $Pog(h)$ is 
named a {\em pre-$Pog(h)$}.
Take a pre-$Pog(h)$ and embed it in the planar disk with center at the origin and
radius $h+1$, denoted $D$, of the usual 
plane so that the auxiliar points are in
the boundary of $D$. The antipodal points of $\partial D$ are identified, 
forming {\em real projective plane} $\mathbb{RP}^2$. In particular 
pairs of antipodal auxiliary points become a single bivalent vertex which 
is removed and the result is the graph $Pog(h)$ embedded into $\mathbb{RP}^2$. 
(see left side of Fig.\ref{Pog2e2emeio})
This completes the definition of $Pog(h)$, in the case of integer $h$. 

\noindent
{\bf Case $h$ is half integer.} If $h$ is a half integer then $Pog(h)$ 
has $\lfloor h \rfloor$ orbits each with $z=4h$ vertices and a degenerated orbit
corresponding to the extra $\frac{1}{2}$ and inducing a single central vertex. In the complex plane the $hz+1$ vertices of
$Pog(h)$ are $\{k \exp(2\pi i j/z): k=1,2,\ldots, \lfloor h \rfloor, \hspace{2mm} j=0,\ldots, 
z-1\} \cup \{0\}.$
The orbital and radial edges as well as the 
identifications are defined similarly as in the case $h$ integer.
The extra ingridient is that there are $z$ edges linking $0$ to the vertices in the 
innermost non-degenerated orbit (see right side of Fig.\ref{Pog2e2emeio}). 

The shapes of the $Pog_h$'s are taylored in such a way that it has $z$ {\em zigzag paths}:
such a path is exemplified in thick edges in Fig.\ref{Pog2e2emeio}). These paths alternates 
choosing the rightmost and leftmost edges at each vertex. Since $\mathbb{RP}^2$ is
non-orientable, in traversing an edge crossing the boundary of $D$ we must repeat 
the direction (left-left or right-right, 
instead of changing it). Note that a zigzag path is closed since it links
two antipodal auxiliary points in $D$ before they are identified in $\mathbb{RP}^2$.


\begin{figure}
\begin{center}
\includegraphics[width=16.5cm]{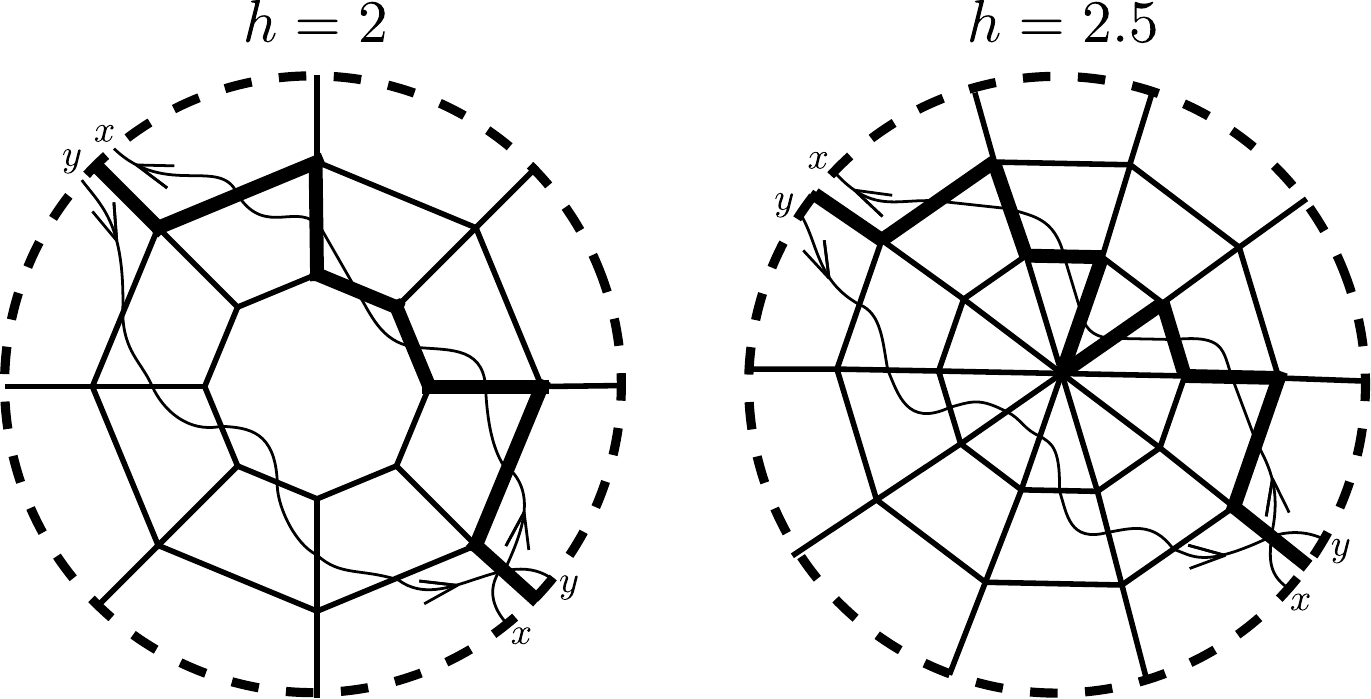}
\caption{\sf Small instances of $Pog_h$, $h=2$, integer and $h=2.5$, 
half integer: $Pog_2 \hookrightarrow 
\mathbb{RP}^2$ and $Pog_{2.5} \hookrightarrow \mathbb{RP}^2$. 
Cellular embeddings of the graphs into 
the {\em real projective plane} or a disk with antipodal 
identification, $\mathbb{RP}^2$. The two thick
closed paths are instances of {\em zigzag paths}. 
There is a total of $z=4h$ zigzag paths in $Pog_h$. 
Closely related to a zigzag path is a {\em closed straight line}, 
which is depicted as a thin line
which goes parallel to an edge 
crossing it at the middle and following close the second half of the edge,
turning at the angle to the next edge, where the process is repeated for all edges of the zigzag path
which are crossed once by the closed straight line. The graph induced by
the closed straight lines of a map is called the line embedding of the {\em phial map}. 
The {\em graph of this embedding} is the one whose vertices are the closed straight lines
of the map and whose edges are the intersection points of two of such lines (which 
may coincide). The line embeddings are in 1-1 correspondence with the usual cellular
embeddings which occur in another surface. This surface is determined, but not really
relevant here for our current purposes. As a crucial property, we have that the graphs
of the line embeddings induced by the $Pog_h$'s are the complete graphs $K_z$, with $z$ even. 
This is straightforward by the circular symmetry of 
these projective graphs: every pair of closed lines
cross exactly once. To obtain $Pog_h$ and its dual as labelled graphs
consistent with the labels of $G_3=K_z$ it is convenient to embed it 
into $S^{23}$. This can be done combinatorially 
by the {\em shaded rozigs}, see Fig. \ref{fig:rozig4}. 
}
\label{Pog2e2emeio}
\end{center}
\end{figure}

\begin{figure}
\begin{center}
\footnotesize{
  \begin{tabular}{|d?b|a|b|a|b|a|b?b|a|b|a|b|a|b|} \hline
  21 & 14 & 61 & 18 & A1 & 1C & E1 & 1G & 31 & 15 & 71 & 19 & B1 & 1D & F1 \\ \hline
  43 & 36 & 83 & 3A & C3 & 3E & G3 & 31 & 53 & 37 & 93 & 3B & D3 & 3F & 23 \\ \hline
  65 & 58 & A5 & 5C & E5 & 5G & 15 & 53 & 75 & 59 & B5 & 5D & F5 & 52 & 45 \\ \hline
  87 & 7A & C7 & 7E & G7 & 71 & 37 & 75 & 97 & 7B & D7 & 7F & 27 & 74 & 67 \\ \hline
  A9 & 9C & E9 & 9G & 19 & 93 & 59 & 97 & B9 & 9D & F9 & 92 & 49 & 96 & 89 \\ \hline
  CB & BE & GB & B1 & 3B & B5 & 7B & B9 & DB & BF & 2B & B4 & 6B & B8 & AB \\ \hline
  ED & DG & 1D & D3 & 5D & D7 & 9D & DB & FD & D2 & 4D & D6 & 8D & DA & CD \\ \hline
  GF & F1 & 3F & F5 & 7F & F9 & BF & FD & 2F & F4 & 6F & F8 & AF & FC & EF \\ \hline
  12 & 23 & 52 & 27 & 92 & 2B & D2 & 2F & 42 & 26 & 82 & 2A & C2 & 2E & G2 \\ \hline
  34 & 45 & 74 & 49 & B4 & 4D & F4 & 42 & 64 & 48 & A4 & 4C & E4 & 4G & 14 \\ \hline
  56 & 67 & 96 & 6B & D6 & 6F & 26 & 64 & 86 & 6A & C6 & 6E & G6 & 61 & 36 \\ \hline
  78 & 89 & B8 & 8D & F8 & 82 & 48 & 86 & A8 & 8C & E8 & 8G & 18 & 83 & 58 \\ \hline
  9A & AB & DA & AF & 2A & A4 & 6A & A8 & CA & AE & GA & A1 & 3A & A5 & 7A \\ \hline
  BC & CD & FC & C2 & 4C & C6 & 8C & CA & EC & CG & 1C & C3 & 5C & C7 & 9C \\ \hline
  DE & EF & 2E & E4 & 6E & E8 & AE & EC & GE & E1 & 3E & E5 & 7E & E9 & BE \\ \hline
  FG & G2 & 4G & G6 & 8G & GA & CG & GE & 1G & G3 & 5G & G7 & 9G & GB & DG \\ \hline
  \end{tabular}
  }
\includegraphics[width=14.cm]{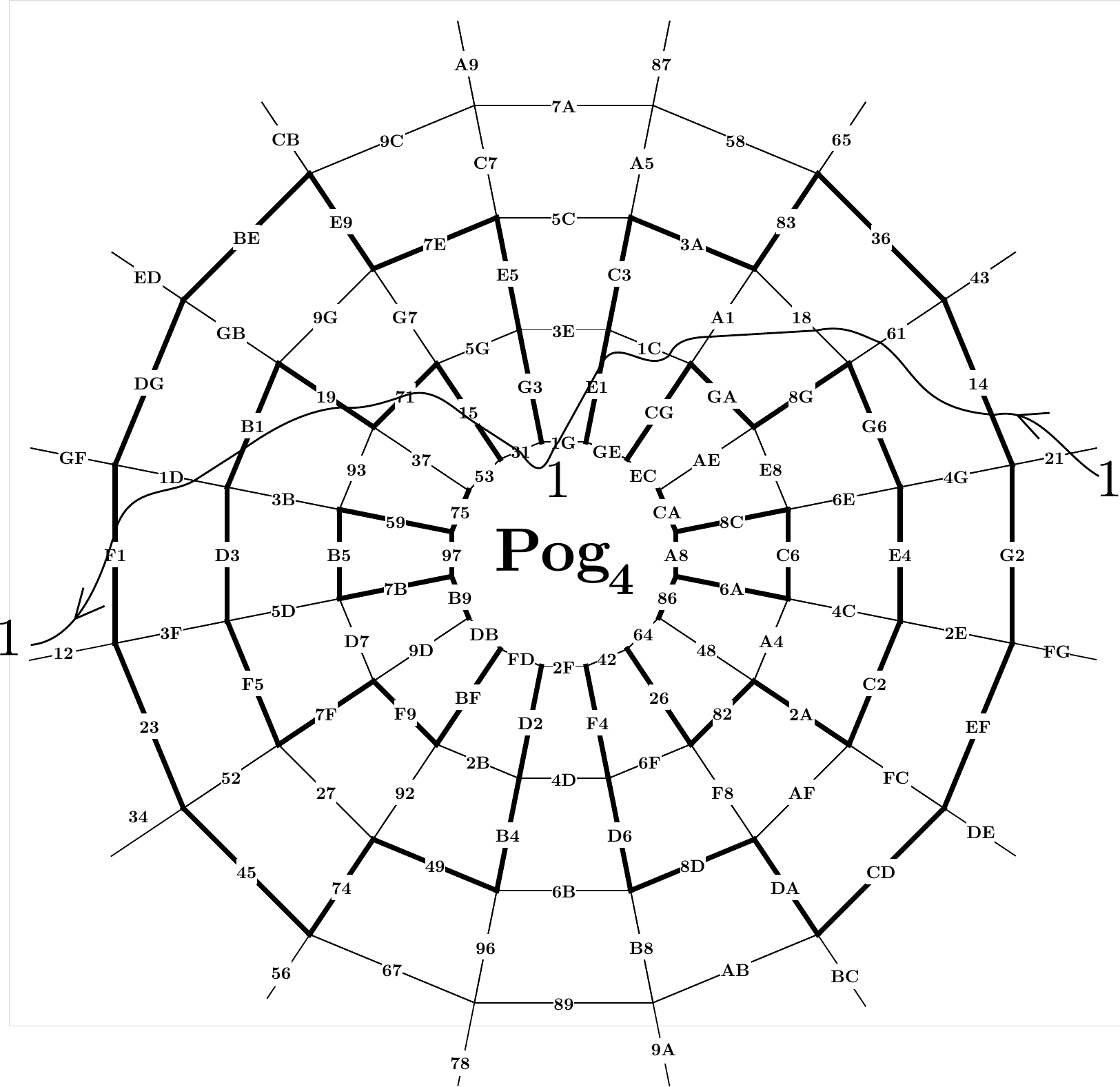} \\
\caption{\sf Example of labelled shaded rooted oriented zigzags or {\em labelled shaded rozigs} 
$G_1=Pog_4$.  We also display a strong $O$-join, denoted by $T$, 
given in thick edges: the parity of the number of edges of $T$ in the 
coboundary (boundary) 
of a vertex (a face) and 
the parity of the degree of the vertex (the face) of $G_1=Pog_4$ 
coincide. The labels of the vertices of 
$K_{16}$ are the digits in $1,2\ldots,9,A,B,C,D,E,F,G$ (base 17). An  
edge of $G_3=K_z$ is labelled by an unordered 
pair of vertices. Note that the face boundaries in clockwise order 
and the vertices coboundaries in counter-clockwise order correspond
to directed polygons in $G_3=K_z$. Rozig 1 is displayed.}
\label{fig:rozig4}
\end{center}
\end{figure}

\section{Combinatorially Constructed Labelled $Pog_h$}
By using a combinatorial construction for $Pog_h$ we get the triad of graphs $G_1=Pog_h$ its dual
in $\mathbb{RP}^2$, $G_2$ and its phial $G_3=K_z$. The construction is based on a table named
{\em shaded rozigs} which amounts to an embedding of $K_z$ into some higher genus surface. We refer to
Fig. \ref{fig:rozig4}. 

The rozig table has $z$ rows and $z-1$ columns. 
Each entry of the table is an ordered distinct 
pair of labels in $\{1,2,\ldots,z$ and each such pair appears 
twice (maybe with the symbols switched). 
These symbols label the vertices of the complete graph and the pair is an oriented form
an edge of $G_3=K_z$.
The filling of the table depends on a simple function
$suc2 \{1,2,\ldots z\} \times \{z\} \longrightarrow \{1,2,\ldots z\}\times \{z\} $,
where $suc2(\ell,z)=\ell+2$, if $\ell \le z-2$, $suc2(z,z)=1$, $suc2(z-1,z)=2$. 

The rozig table 
has 3 types of columns: the {\em projective column}, formed by the 0-column,
the {\em left columns}, formed by columns 1 to $z/2$ and the 
{\em right columns} formed by columns $z/2+1$ to $z-1$. 

\vspace{3mm}
\noindent
{\bf Defining the first row of the rozig table.}
The entries in the first row start with (2,1) in the projective column, followed by
$$(1,4), (6,1), \ldots, (z,1)
{\rm\ or\ by\ } 
(1,4), (6,1), \ldots, (1,z),$$ 
according to $z\equiv 2 \mod 4$ or $z\equiv 0 \mod 4$
filling the left columns. Finally we have,  if $z\equiv 2 \mod 4$,
$$(1,3), (5,1), \ldots, (z-1,1) 
{\rm\ or\ by\ } 
(3,1), (1,5), \ldots, (1,z-1),$$
filling the right columns. This completes the filling of the first row of the rozig table.
This row corresponds to the cyclic order of the oriented edges of the coboundary of
vertex 1 of $G_3=K_z$. It corresponds also to a rooted oriented zigzag (rozig) path 
labelled 1 in $G_1=Pog_h, z=4h$. See Fig. \ref{fig:rozig4}.

\vspace{3mm}
\noindent
{\bf Defining the other rows of the rozig table.}
To get row $i+1$ from row $i$ in the rozig table just apply $suc2$ to the individual symbols
of the pairs. This completes the definition of rozig table. 
From its rows we get a {\em rotation} for $K_z$,
namely a cyclic ordering for the edges incident to each vertex $i$ of $K_z$. 

\vspace{3mm}
\noindent
{\bf Yet another combinatorial counterpart for graphs embedded into surfaces.}
To obtain a combinatorial counterpart for an embedding of a graph we need a rotation
(which we have: the rows of the rozig) together with the corresponding {\em twist} 
which is the subset of edges that are twisted
for the fixed rotation. In our case, the twisted edges 
are the ones which correspond to the radial edges
of $Pog_h$. The non-twisted ones correspond to the orbital edges of $Pog_h$. In terms of rozigs,
a twisted edge is one traversed in opposite directions by the two zigzags that traverse the edge.
The pair {\bf (rotation,twist)} is sufficient to describe the embedding because 
from it we can recover the entire $Q$-graph: 
given an immersion respecting the rotation of $Q$ 
(with crossings between the 1-colored edges) in the plane, given a twisted 
edge $e$ the pair of edges of color 2 in the hyperedge of $Q$ corresponding to $e$ is 
replaced by the crossing edges.

\vspace{3mm}
\noindent{\bf The relevance of the shading.}
All the edges in a column of the rozig table are radial or all are orbital. We can shade the 
columns so that an edge is twisted in the rotation iff it is in a shaded column. In this way,  
shading defines the twist of the map and complements the rozigs completing its 
combinatorial presentation.

\vspace{3mm}
\noindent
{\bf Defining the shading}. 
The projective column is shaded,  
the left columns alternate (non-shaded, shaded) starting with non-shaded. The right columns
are shaded or not according to the reflexion of the left columns in the vertical line
separating the left and right columns. See Fig. \ref{fig:rozig4}.

\section{Linear Models for MinStrongOjoin and MaxCut}
Suppose that $G_1=Pog_h$ and $G_2)$ are duals in $\mathbb{RP}^2$ and 
$G_3=K_z$ $(z=4h)$ embedded in some surface as the phial of $G_1$. 
The common set of edges is denoted $E$. In order to prove that
$G_3$ is rich is enough to prove that $G_1$ is rich. We have that 
$\dim (\mathcal{V}_1^\perp/\mathcal{F}_1)=(|E|-v+1)-(f-1)=|E|-v-f+2=-\chi+2=1$, 
since $\chi(\mathbb{RP}^2)=1$. Any zigzag in $\mathcal{Z}_1$ can be adjoined to 
$\mathcal{F}_1$ to generate the cycle space of $G_1$. Note that each zigzag is an
orientation reversing polygon, so it is not in the span of the boundaries of the faces.
Thus $G_1$ is rich, whence $G_3$ is rich.

\vspace{3mm}
\noindent
{\bf Triangles and quadrangles in $V_{12}$ spanning the cycle space of $K_z$.}
Denote by $V_{12}$ the set of polygons $p$ of length 3 and 4 of 
graph $G_3=K_z$ which corresponds to 
the coboundary of the vertices of $G_1$ and $G_2$. We have 
$\langle V_{12} \rangle$=$\mathcal{V}_3^\perp$, because at most one polygon 
(correponding to the central face if $z\equiv 0 \mod 4$ or the central vertex if
$z \equiv 2 \mod 4)$ has number of  sides distinct from 3 and 4. Note that this 
polygon is equal to the sum of all the other polygons (3- and 4-gons) 
in the same $G_i$.

We can now define the first of our polytopal models. It has a variable $x_e' \in \mathbb{R}^{|E|}$ 
for each $e \in E$ and a variable $s_p \in \mathbb{R}^{|V_{12}|}$  for each $p \in V_{12}$.

\begin{center}
 $
 P_0' = \left\{
\begin{array} {ll}
p\in V_{12}: & 2 s_p + \sum\{x_e':e \subseteq p\} = |p|\\
\rm {bounds:\ } & 0 \le x_e' \le 1\ \forall e' \in E, s_p \ge 0, \forall p \in V_{12}.
\end{array}
\right.
$
\end{center}

\begin{proposition}
 $P_0'$ is a linear model for the MinStrongOJoin problem.
\end{proposition}
\begin{proof}
 Any characteristic vector of a strong $O$-join satisfies 
 the linear restrictions of $P_0'$.
 Reciprocally, if $(x_e',s_P)$ is all integer and satisfy 
 these restrictions it is the characteristic vector of a strong $O$-join. 
 \end{proof}
  
  \vspace{3mm}
  \noindent
 {\bf Double slack variables.} Observe that each $s_p$ appears once with coefficient 2. 
 Therefore $\frac{s}{2}$ is 
 a slack variable and $s$ is called a {\em double slack variable}.
 
 \vspace{3mm}
 \noindent
 {\bf Valid inequalities.}
 A {\em valid inequality} for a polytope is one which does not remove any 
 of its points with all integer coordinates. It is straighforward to show that 
 a linear model for a combinatorial problem remains so if we add valid inequalities.
 A class of valid inequalities will be added to $P_0'$ which permits the 
 elimination of the double slack variables $s_p$ and of 
 the unitary upper bounds $x_e' \le 1$.
 
 \begin{figure}
\begin{center}
\label{fig:vtValIneqs}
\includegraphics[width=12cm]{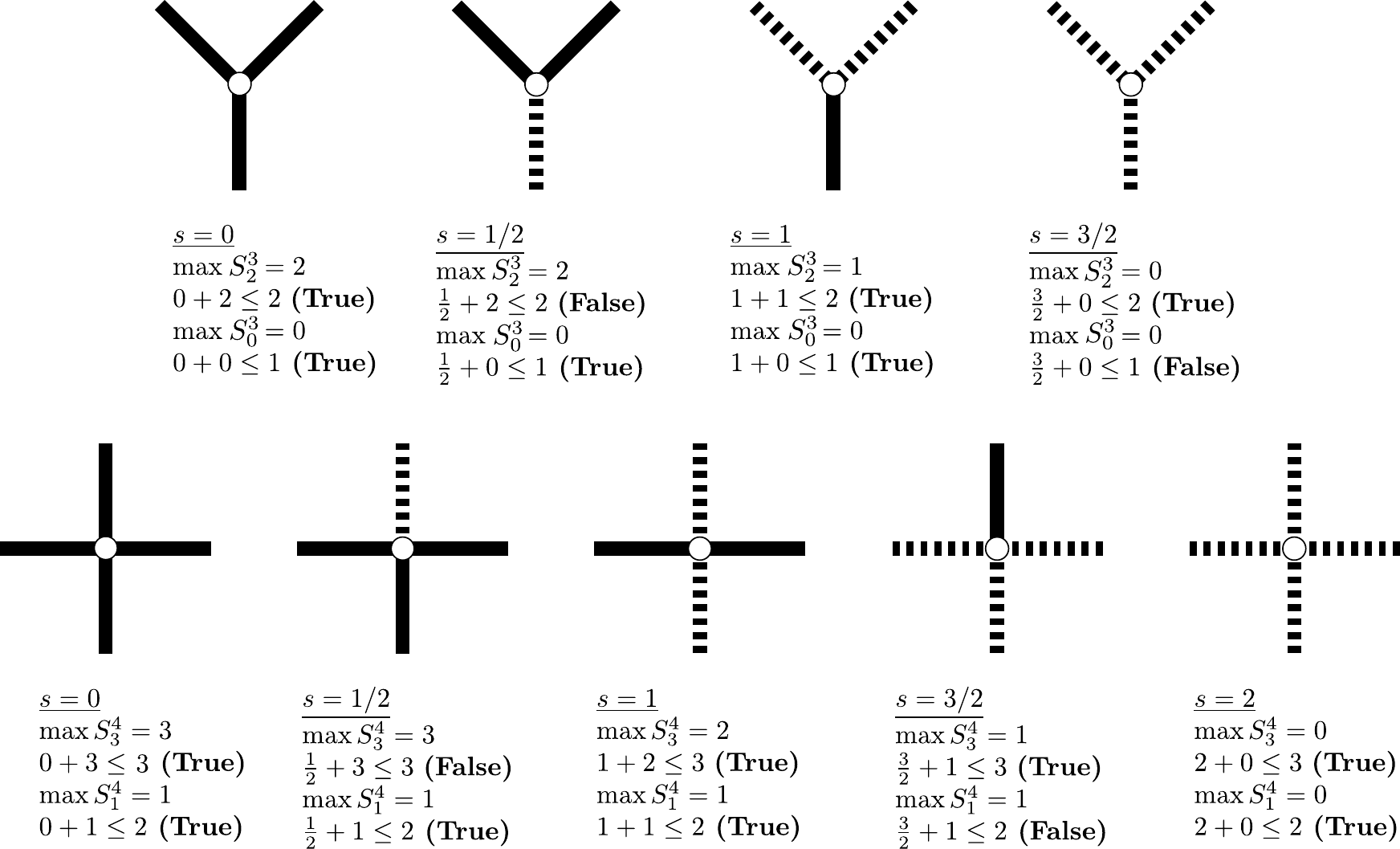} \\
\caption{\sf Valid/violating $pq$-inequalities: $s+\max S^p_q \leq \frac{|p|+|q|-1}{2}$, 
with $|p|$ and $|q|$ of distinct parities where $q \subset p$.  We impose $x'_e \in \{0,1\}$.
The condition characterizing a strong $O$-join can be localized to 
a neighborhood of a generic $3$-vertex in $V_{12}$.
The thick full edges are in $q$ and their set constitutes a 
strong $O$-join. The coboundary of $v$ is $p$, which is a polygon in $G_3$. 
The dashed edges are the ones in $p\backslash q$.
For each half integer $s$ one of its $pq$-inequalities
is violated. If $s$ is integer all its
induced $pq$-inequalites are valid. Thus, integer $x'_e$, $e \in E$, plus 
$pq$-inequalities imply integer double slackness variables $s_p$, $p \in V_{12}$.} 
\label{fig:thetas}
\end{center}
\end{figure}

 Let $p \in V_{12}$ and $q \subset p$ so that $|p|+|q|$ is odd. 
 The {\em $pq$-inequality} is 
 $$ s_p+ \sum\{x_e: e \subset q \subset p\} \le \frac{|p|+|q|-1}{2}.$$ 
 The following theorem is central in this work.
 \begin{theorem}
  The $pq$-inequalities eliminate fractional double slack variables 
  $s_p$ in the sense that after including them, 
  integer $x_e', e' \in E$  imply integer $s_p, p \in V_{12}$. 
 \end{theorem}
\begin{proof}
 Let $S^p_q = \sum\{x_e': e \subset q \subset p\}$. 
The analysis for (0-1)-integers $x_e'$ given in Fig. \ref{fig:vtValIneqs} shows 
that the vertices with a fractional
$s_p$ are precisely the ones that violate some restriction. 
The neighborhood of a vertex in $V_{12}$. 
The thick edges have $x_e'=1$ and the dashed
edges have $x_e'=0$. The whole coboundary of the vertex is the edge set of a polygon $p \in G_3$.
\end{proof}

By simply adding the $pq$-inequalities provides another linear 
model for the MinStrongOjoin problem:

\begin{center}
 $
 P_1' = \left\{
\begin{array} {ll}
p\in V_{12}: & 2 s_p + \sum\{x_e':e \subseteq p\} = |p|\\
q \subset p\in V_{12}: &  s_p + \sum\{x_e':e \subset q \subset p\} \le \frac{|p|+|q|-1}{2}\\
\rm {bounds:\ } & 0 \le x_e' \le 1, s_p \ge 0.
\end{array}
\right.
$
\end{center}
 
Since integrality of $x_e'$ imply integrality of the $s_p$ and each of these appears once
in an equation, we can dispose of these double slackness variables 
variables by considering its implicit definition,
$$ s_p (integer) := \frac{|p| - \sum\{x_e':e \subseteq p\}}{2} {\rm \ \ we \ obtain}, $$
\begin{center}
 $
 P_2' = \left\{
\begin{array} {ll}
q \subset p\in V_{12}, |p|-|q| {\rm \ odd\ }: &  |p| - \sum\{x_e':e \subseteq p\}
+ \sum\{2x_e':e \subset q \subset p\} \le |p|+|q|-1\\
\rm {bounds:\ } & 0 \le x_e' \le 1 \hspace{3mm} \forall e \in E.
\end{array}
\right.
$
\end{center}

Consider
$q \subset p\in V_{12}$, $|p|-|q|$ odd:  
$$|p| - \sum\{x_e': e \subseteq p\}
+ \sum\{2x_e':e \subset q \subset p\} \le |p|+|q|-1$$
$$ \Longleftrightarrow |p| - \sum\{x_e':e \subset p\backslash q\} 
- \sum\{x_e':e \subseteq q \subset p\}
+ \sum\{2x_e':e \subset q \subset p\} \le |p|+|q|-1$$
$$ \Longleftrightarrow \sum\{x_e':e \subseteq q \subset p\} - \sum\{x_e':e \subset p\backslash q\} 
\le |q|-1.$$

 \begin{theorem}
 The polytope\\
 \begin{center}
 $
 P_2' = \left\{
\begin{array} {ll}
q \subset p\in V_{12}, |p| - |q| {\rm \ odd} : &  \sum\{x_e':e \subset q \subset p\}
- \sum\{x_e':e \subset p\backslash q\} \le |q|-1\\
\rm {bounds:\ } & 0 \le x_e' \le 1 \hspace{3mm} \forall e \in E.
\end{array}
\right.
$
\end{center}
is a linear model for the MinStrongOjoin problem.
\end{theorem}
\begin{proof}
 It is straightforward from the equivalences above.
\end{proof}

We want to get a linear model for the MaxCut problem. Given $P_2'$ it is enough to replace each 
variable $x_e'$ by $x_e=1-x_e'$. This has the effect of complementing the 
characteristic vectors and the minimization problem becomes a maximization one. We get\\
$$ \sum\{x_e':e \subseteq q \subset p\} - \sum\{x_e':e \subset p\backslash q\} 
\le |q|-1$$
$$ \Longleftrightarrow \sum\{1-x_e: e \subseteq q \subset p\} - \sum\{1-x_e:e \subset p\backslash q\} 
\le |q|-1$$
$$ \Longleftrightarrow |q|-\sum\{x_e: e \subseteq q \subset p\} - 
(|p|-|q|) + \sum\{x_e:e \subset p\backslash q\} 
\le |q|-1$$
$$ \Longleftrightarrow  \sum\{x_e:e \subset p\backslash q\}  
-\sum\{x_e: e \subseteq q \subset p\} 
\le |p|-|q|-1$$
$$ \Longleftrightarrow  \sum\{x_e:e \subset p\backslash q\}  
-\sum\{x_e: e \subseteq q \subset p\} 
\le |p| - p^- -1$$
$$ \Longleftrightarrow  \sum\{x_e:e \subset p\backslash q\}  
-\sum\{x_e: e \subseteq q \subset p\} 
\le p^+ -1.$$

\vspace{3mm}
\noindent
{\bf Sign of an edge in a polygon.}
Edges in
$p\backslash q$ have sign +1 and edges in $q$ have sign -1. 
Let $\sigma^e_p$ be the sign of edge $e$ in polygon $p$. 
Let $p^+$ and $p^-$ denote respectively the number of $+1$ signs and $-1$ signs on
the edge variables of the polygon $p$.
Note that $|p|-|q|$ odd $ \Rightarrow $  $|p|-p^-$ odd $\Rightarrow p^+$ odd.
Then the last equivalence can be rewritten as
$$ \sum\{\sigma^e_p x_e:e \subset p\} \le p^+ -1.$$ 
Let $S^\pm_{12}$ denote the polygons in $V_{12}$ 
arbitrarily signed except for the fact that
$p^+$ is odd. Note that we have disposed the $q$'s by using signed polygons. 
In the Theorem below the linear restrictions forming the polytope 
are induced by signed forms of
the coboundaries of the vertices and the signed forms of the boundary of the faces of map
$G_1=Pog_h \hookrightarrow \mathbb{RP}^2$. The phial graph of $G_1$ is $G_3=K_z, z=4h$,
embedded into some higher genus surface $S^{23}$, which does not concern us except 
for the practical fact that $G_3=K_z \hookrightarrow S^{23}$ via shaded rozigs 
is the easier way to obtain combinatorially the graphs 
$G_1=Pog_h$ and its dual $G_2$ in $\mathbb{RP}^2$ so that
$G_1, G_2,G_3$ have the same edge set $E$.

 \begin{theorem}
 \label{theo:modelMaCut}
  The polytope\\
 \begin{center}
 $
\mathbb{P}_{12} = \left\{
\begin{array} {ll}
p\in S^\pm_{12}: &  \sum\{\sigma^e_p x_e:e \subset p\} \le p^+ -1\\
\rm {bounds:\ } & x_e \ge 0 \hspace{3mm} \forall e \in E.
\end{array}
\right.
$
\end{center}
is a linear model for the MaxCut problem on the complete graph $K_z, z$ even. 
\end{theorem}
\begin{proof}
 It is straightforward from the equivalences above, except for the unitary upper bounds.
 Given any $ij \in E$ there is in $G_1$ in a coboundary of a 
vertex or the boundary of a face of degree 3 or 4 containing $ij$.
The variables correspond to unoriented edges. So we have 
$x_{ij}=x_{ji}$ for every pair of distinct vertices of $G_3=K_z$.

\vspace{3mm}
\noindent
{\bf Case 3.}
In the first case there is a $k$ so that 
$x_{ij} + x_{jk} + x_{ki} \le 2$ and $x_{ij}-x_{jk}-x_{ki} \le 0$.
Adding these, $2x_{ij} \le 2$ or $x_{ij}\leq 1$.

\vspace{3mm}
\noindent
{\bf Case 4.}
If $ij$ is in the coboundary of a vertex or in the boundary of a face of degree 4, there
are $k$ and $l$ so that 
$ 
x_{ij}-x_{jk}-x_{kl}-x_{li}\leq 0, \hspace{2mm} 
x_{ij}+x_{jk}+x_{kl}-x_{li} \leq 2 \hspace{2mm}.  
$
Adding we get $2 x_{ij} - 2 x_{li}\le 2$ or $x_{ij} - x_{li} \le 1.$
We also have
$ 
x_{ij}+x_{jk}-x_{kl}+x_{li}\leq 2, \hspace{2mm} 
x_{ij}-x_{jk}+x_{kl}+x_{li} \leq 2 \hspace{2mm}.  
$
Adding we get $2 x_{ij} + 2 x_{li}\le 4$ or $x_{ij} + x_{li} \le 2.$
The inequalities imply that $2 x_{ij} \leq 3$ 
and, since $x_{ij}$ is an integer,
$x_{ij} \leq 1$. 
\end{proof}

\vspace{3mm}
\noindent
{\bf Estimating $|S^\pm_{12}|$.}
For this estimation we count the number of 
3-vertices, 4-vertices, 3-faces and 4-faces of $G_1$.
If $h$ is an integer, then the number of 3-vertices is $z=4h$ and the number of 
3-faces is 0. The number of 4-vertices of $G_1$ is $z(h-1)$. The number of 4-faces
is $z(h-1)+z/2$. 
If $h$ is a half integer, then the number of 3-vertices is 0, the number of 3-faces is
$z=4h$. The number of 4-vertices is $\lfloor h \rfloor z$. The number of 4-faces is 
$(\lfloor h \rfloor-1)z+z/2$. 

\vspace{3mm}
\noindent
{\bf Unifier of vertices and faces}.
Let a {\em unifier} be either a vertex or a face or $G_1$.
If $h$ is integer the number of
$3$-unifiers is $z$ and the number of 4-unifiers is 
$z(h-1) + z(h-1)+z/2=2z(h-1)+z/2$. If $h$ is a half integer, then the number of
3-unifiers is $z$ and he number of 4-unifiers is  
$\lfloor h \rfloor z +  (\lfloor h \rfloor-1)z+z/2 = 2\lfloor h \rfloor z - z/2$. 

\vspace{3mm}
\noindent
{\bf Cardinality of $S_{12}^\pm$ in terms of unifiers.} 
This cardinality is 4 times the number of $3$-unifiers 
plus 8 times the number of 4-unifiers of $G_1$. Thus, if $h$ is an integer,
the $|S_{12}^\pm|$ is $4z+8(2z(h-1)+z/2)=8z+16 zh - 16z = 16 zh - 8z \le 16 zh=4z^2.$
If $h$ is a half integer, then $|S_{12}^\pm|$ is $4z+8( 2\lfloor h \rfloor z - z/2)= 
16 \lfloor h \rfloor z \le 16 z h=4z^2$.
Thus, in every case, $|S_{12}^\pm| \le 4z^2=O(|E|).$
In fact we have $4z^2\le 10 |E| = 5 z^2-5 z 
\Longleftrightarrow 5z \le z^2$ 
$\Longleftrightarrow z \geq 5,$ which is clearly true, since there is no use in 
working with $K_4$. 

\begin{theorem}
 The number of linear inequalities defining $P_2$ is at most $11|E|$. Each of
 them involves the sum of no more than 4 edge variables 
 with $\pm 1$ coefficients. The right hand side 
 of them is either 0, 1 or 2.  
\end{theorem} 
\begin{proof}
We have established in the above discussion that 
$|S_{12}^\pm| \leq 10|E|$. There are $|E|$
inequalities corresponding to the non-negativity of the variables. 
The bounds on each inequality are directly seen to hold.
So the result follows.
\end{proof}

\section{Conclusion}
The first author acknowledges the partial financial support of CNPq-Brazil, process number 
302353/2014-3. The second author acknowledges the financial support of FACEPE, IBPG-1295-1.03/12.
In a companion paper the authors show how to use Theorem \ref{theo:modelMaCut}
to improve considerably the running time of the IP-solver
 SCIP (\cite{Achterberg04scip, GamrathFischerGallyetal.2016}), working in the same MaxCut Problem, using the $\mathbb{P}_{12}$-model. 
Moreover each solution
provided by our algorithm (\cite {linshenriques2})
is exact and could be polynomially verifiable 
(polynomial in terms of the number of leaves in the set of 
{\em SSS-trees}). This acronym accounts for Sufficient Search Space Trees, a special 
set of trees).
These trees organize the computation and provide a proof that a solution is complete and correct.
In fact we use it to verify a solution produced by the solver. 
There is also, due to the simplicity of the model,
a number of interesting questions currently under investigation, 
involving theoretical and applied issues. 

\bibliographystyle{plain}
\bibliography{bibtexIndex.bib}

\begin{thebibliography}{1}

\bibitem{Achterberg04scip}
Tobias Achterberg.
\newblock Scip -- a framework to integrate constraint and mixed integer
  programming.
\newblock Technical report, 2004.

\bibitem{GamrathFischerGallyetal.2016}
Gerald Gamrath, Tobias Fischer, Tristan Gally, Ambros~M. Gleixner, Gregor
  Hendel, Thorsten Koch, Stephen~J. Maher, Matthias Miltenberger, Benjamin
  M{\"u}ller, Marc~E. Pfetsch, Christian Puchert, Daniel Rehfeldt, Sebastian
  Schenker, Robert Schwarz, Felipe Serrano, Yuji Shinano, Stefan Vigerske,
  Dieter Weninger, Michael Winkler, Jonas~T. Witt, and Jakob Witzig.
\newblock The scip optimization suite 3.2.
\newblock Technical Report 15-60, ZIB, Takustr.7, 14195 Berlin, 2016.

\bibitem{garey1979computers}
M.R. Garey and D.S. Johnson.
\newblock {\em {Computers and intractability}}.
\newblock Freeman San Francisco, 1979.

\bibitem{goemans1995iaa}
M.X. Goemans and D.P. Williamson.
\newblock {Improved Approximation Algorithms for Maximum Cut and Satisfiability
  Problems Using Semidefinite Programming}.
\newblock {\em Journal of the Association for Computing Machinery}, 1995.

\bibitem{Lins1980}
S.~Lins.
\newblock {\em Graphs of maps, Available in the arXiv as math.CO/0305058}.
\newblock PhD thesis, University of Waterloo, 1980.

\bibitem{Lins1982}
S.~Lins.
\newblock Graph encoded maps.
\newblock {\em Journal of Combinatorial Theory, Series B}, 32:171--181, 1982.

\bibitem{linshenriques2}
S.~Lins and D.~Henriques.
\newblock {The SSS-Algorithm: an exact verifiable algorithm for the MaxCut
  Problem, in Preparation}.
\newblock 2016.

\end{thebibliography}
\end{document}